\documentclass[11pt,oneside]{amsart}
\usepackage{palatino}
\usepackage{graphicx}
\usepackage{amsmath,amsthm,amssymb}
\usepackage{fancyvrb}
\usepackage{url}
\usepackage[height=8.5in,width=6.0in]{geometry}
\usepackage[T1]{fontenc}  
\newtheorem{theorem}{Theorem}

\newtheorem{lemma}[theorem]{Lemma}

\newcommand{\ra}{\rightarrow}
\newcommand{\sgn}{\mbox{sgn}}
\newcommand{\vol}{\mathrm{vol}}
\newcommand{\pr}{\mathrm{P}}

\title[Exact Probabilities in the balanced uniform model]{Exact Probabilities for Three and Four Dice in the Balanced Uniform Model of 3-Sided Dice}
\author{Kent E. Morrison}
\address{American Institute of Mathematics, 600 East Brokaw Road, San Jose, CA 95112}
\email{morrison@aimath.org}
\date{March 8, 2023}
\begin{document}
\maketitle

\begin{abstract}
\noindent We determine the exact probabilities of the different isomorphism classes of tournaments that result from random sets of three and four independent dice drawn from the balanced uniform model of $3$-sided dice. 
\end{abstract}
\large
\renewcommand{\baselinestretch}{1.2}   
\normalsize

\section{Background and Results}
By now it is well-recognized that various models of generalized dice exhibit surprisingly high rates of intransitivity for randomly chosen dice. Given two dice $A$ and $B$, where $A$ has face values $a_1,\ldots,a_n$ and $B$ has face values $b_1,\ldots,b_m$,  we say that $A$ \textbf{dominates} $B$, denoted by $A \succ B$, if it is more likely for $A$ to show a higher value than $B$, that is
\[  \sum_{i,j} \sgn(a_i - b_j) > 0. \]
If the sum is $0$, then neither die dominates the other and we say that the result is a tie. In this article we will always be working with dice with the same number of faces; i. e., $m=n$.

In 2016 Conrey, Gabbard, Grant, Liu, and Morrison  \cite{CGGLM2016} considered $n$-sided dice as integer multisets of size $n$ with elements in $\{1,2,\ldots,n\}$ and satisfying $\sum a_i = n(n+1)/2$. We made two conjectures about random sets of three $n$-sided dice as $n \ra \infty$.
\begin{itemize}
\item The probability of any ties goes to zero. 
\item The probability of an intransitive triple goes to $1/4$.
\end{itemize}
For three dice $A,B,C$ and in the absence of ties there are $2^3$ possible configurations for the relations between the three pairs $(A,B)$, $(B,C)$, and $(A,C)$. Two of the eight configurations represent the intransitive cycles $A \succ B \succ C \succ A$ and $A \succ C \succ B \succ A$. The other six configurations are transitive chains.  Evidence for the conjecture was provided by Monte Carlo simulations showing that the probabilities of all eight configurations appear to approach $1/8$, and so the probability of an intransitive triple approaches $1/4$ and the probability of a transitive triple approaches $3/4$. 
We also made much more speculative generalized conjectures concerning $k$ random dice for any fixed positive integer $k$. First, the probability that there is a tie between any of the $k$ dice goes to $0$. With no ties there are $2^{\binom{k}{2}}$ outcomes for all pairwise comparison, and each of these outcomes is represented by a complete directed graph on $k$ vertices,  i.e., a \emph{tournament}. The second generalized conjecture is that in the limit all the tournaments are equally probable.

A few years later these conjectures became the topic of a Polymath project. An alternative model was introduced, the \emph{balanced sequence model}, in which an $n$-sided die is a sequence $A=(a_1,\ldots,a_n)$ of elements in $\{1,2,\ldots,n\}$ with total $n(n+1)/2$. The balanced sequence model is easier to work than the multiset model. For example, a random  die is now a sequence of $n$ iid random variables uniform in $\{1,\ldots,n\}$ and conditioned  on the sum being $n(n+1)/2$. The main results in \cite{Polymath2022} are that the conjectures for three dice hold for the balanced sequence model. Thus, the questions for three dice in the multiset model are still open, but it would be a surprise if they turn out to be false.

At the same time the Polymath project spurred some other work, and there is strong computational evidence that the generalized conjecture about tournaments fails with four dice. If the $2^{\binom{4}{2}}$ tournaments were all equally probable, then the probability of a transitive chain and the probability of an intransitive cycle would both have limit $3/8=0.375$, but several independent simulations done for both the multiset and the sequence models produce data with the experimental probabilities somewhat greater than $0.38$.

Further evidence appears in the results of Cornacchia and H\k{a}z{\l}a \cite{CornacchiaHazla2020} for four dice in the balanced uniform model. In this model a random $n$-sided die is a point $(a_1,\ldots,a_n) \in [0,1]^n$ chosen uniformly and conditioned on $\sum_i  a_i = n/2$. They prove that there exists $\varepsilon > 0$ such that for $n$ sufficiently large, the probability is greater than $3/8 + \varepsilon$ that four random dice have a transitive tournament. In the same paper they prove that, in fact, the conjectures do hold for three dice in the balanced uniform model. That is, as $n \ra \infty$, the probability that three $n$-sided dice form an intransitive triple approaches $1/4$, and the probability that they form a transitive chain approaches $3/4$. 

In this note we consider random sets of three dice and four dice at the other extreme with $n=3$, which is the least number of sides of any interest. With three dice there are eight tournaments in two isomorphism classes. One isomorphism class consists of the six transitive chains and the other class contains the two intransitive cycles. We find the the transitive probability is exactly $973/1230=0.76015625$ and the intransitive probability is exactly $307/1280=0.23984375$. With four dice the 64 tournaments fall into four different isomorphism classes: the 24 completely transitive chains, the 24 intransitive cycles, the 8 tournaments with an overall winner and a 3-cyle, and the 8 tournaments with an overall loser and a 3-cycle. We find the probabilities are:
\begin{align*}
 \text{transitive} \quad  \frac{110413771}{25804800}  &\approx 0.42788 \\
 \\
 \text{intransitive}\quad   \frac{99930571}{25804800} &\approx 0.38726 \\
 \\
 \text{winner + 3-cycle} \quad  \frac{23851829}{25804800} & \approx 0.09243 \\
 \\
 \text{loser + 3-cycle}\quad   \frac{23851829}{25804800}  &\approx 0.09243
\end{align*}

\section{Three Dice}
The dominance relation between $A$ and $B$ is invariant under permutations of the coordinates, and so we can assume that $a_1 \le a_2 \le a_3$. We also have $a_1+a_2+a_3=3/2$. The sample space for a single die is the polygonal region in the $a_1  a_2$-plane defined by the inequalities
\[ 0 \le a_1, \quad a_1 \le a_2, \quad 1 \le 2a_1 + 2a_2 , \quad 2a_1 + 4a_2 \le 3 .\]
The third and fourth inequalities come from $a_3 \le 1$ and $a_2 \le a_3$ by replacing $a_3$ with $3/2 - a_1 -a_2$. 
Call this region $Q$. Its boundary is a quadrilateral within the unit square and it has area $1/8$. The probability measure on $Q$ is the normalized area. See Figure 1. 

\begin{figure}[h]
\begin{center}
\includegraphics[width=3.0in]{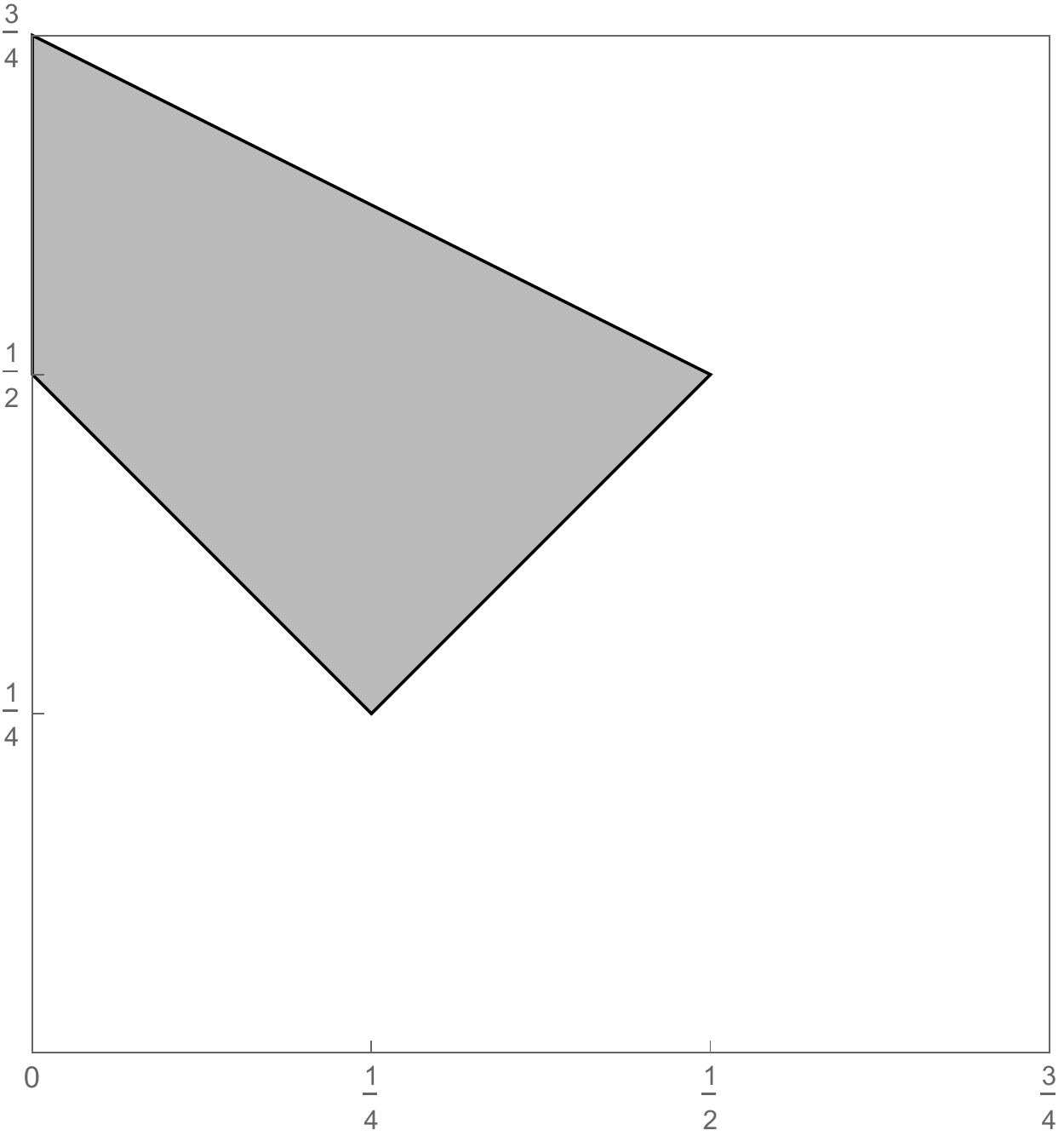}
\caption{The sample space $Q$ for a single random die}
\label{feasible}
\end{center}
\end{figure}

The sample space for three dice $(A,B,C)$ is the six-dimensional polytope $Q^3 \subset [0,1]^6$. It is defined by $12$ linear inequalities, four for each of the three dice. The volume of $Q^3$ is $1/8^3=1/512$. We use the coordinates $(a_1,a_2,b_1,b_2,c_1,c_2)$ for points in $Q^3$.

Our goal is to compute the volume of the subset of $Q^3$ defined by $A \succ B \succ C \succ A$. Let's first consider the inequalities that define $A \succ B$. It is not necessary to check all 9 comparisons between each $a_i$ and $b_j$. It is sufficient to simply compare $a_i$ with $b_i$.
\begin{lemma}
There are three different configurations for which $A \succ B$. They are
\begin{align*}
 (\succ_1):  a_1 < b_1, \quad a_2 > b_2, \quad a_3 > b_3 \\
 (\succ_2):  a_1 > b_1, \quad a_2 < b_2, \quad a_3 > b_3 \\
 (\succ_3):  a_1 > b_1, \quad a_2 > b_2, \quad a_3 < b_3
\end{align*}
\end{lemma}
\begin{proof}
If $a_i > b_i$ for all $i=1,2,3$, then $\sum a_i > \sum b_i$, which is impossible. However, if any two of the three hold, then $A$ dominates $B$. The configurations are labeled $\succ_i$ according to which inequality $a_i > b_i$ fails to hold.
\end{proof}

\begin{lemma}
If $A \succ_i B$ and $B \succ_i C$, then $A \succ_i C$. (That is, the relations $\succ_i$ are transitive.)
\end{lemma}

\begin{proof} Straightforward. Just write down the inequalities.
\end{proof}

\begin{lemma}
If $A, B, C$ form an intransitive triple with $A \succ B \succ C \succ A$, then $A \succ_i B \succ_j C \succ_k A$ where $\{i,j,k\} = \{1,2,3\}$. 
\end{lemma}
\begin{proof}
Assume that the three relations are not distinct. By relabeling, if necessary,  we can assume that $i=j$. Then $A \succ_i B$ and $B \succ_i C$, and therefore $A \succ_i C$, which is a contradiction.
\end{proof}

For a permutation $\sigma=\sigma_1 \sigma_2 \sigma_3$ define $E_\sigma$ to be both the event 
\[ A \succ_{\sigma_1} B \succ_{\sigma_2} C \succ_{\sigma_3} A \]
and the subset of $Q^3$ representing the event.
Each $E_\sigma$ is a polytope defined by nine additional inequalities, three for each of the three relations. The union $E = \bigcup E_\sigma$ is the event that $A \succ B \succ C \succ A$, and $\pr(E)=\vol(E) /\vol(Q^3)$. Also, $\vol(E) = \sum_\sigma \vol(E_\sigma)$.
The volumes of the $E_\sigma$ are computed using the SageMath interface to LattE integrale. 

 \begin{lemma} 
 \begin{align*} 
 \pr(E_\sigma) &= 
      \begin{cases}
            \displaystyle \frac{23}{1800} & \text{if $\sigma = 123, 231,  312$} \\
              & \quad \\
             \displaystyle \frac{3133}{115200} & \text{if $\sigma = 132, 213, 312$}
       \end{cases} \\
       &\quad \\
  \pr(E) &= \frac{307}{2560} 
  \end{align*}    
 \end{lemma}
 \begin{proof}
   See Section 4 for the code. Note that it suffices to compute the volumes of $E_{123}$ and $E_{132}$ because $E_\sigma$ and $E_{\sigma'}$ have the same volume if $\sigma'$ is a cyclic permutation of $\sigma$. The reason is that the cyclic permutation $(A,B,C) \mapsto (B,C,A)$ is volume preserving and  maps $E_{\sigma_1 \sigma_2 \sigma_3}$ to $E_{\sigma_2 \sigma_3 \sigma_1}$.

 \end{proof}
 
 \begin{theorem}
  The probability that $A, B, C$ form an intransitive triple is $307/1280 = 0.23984375$. The probability that they form a transitive chain is $973/1280=0.76015625$.
 \end{theorem}
 \begin{proof}
 By symmetry $\pr(A \succ B \succ C \succ A)=\pr(A \prec B \prec C \prec A)$, which is $\pr(E)$, and so the probability of an intransitive triple is  $2 \pr(E)$. The transitive probability is the complementary probability $1-2 \pr(E)$.
  \end{proof}
 
\section{Four Dice}
 
 The sample space for four dice $(A,B,C.\,D)$ is the eight-dimensional polytope $Q^4 \subset [0,1]^8$. It is defined by $16$ linear inequalities, four for each of the four dice. The volume of $Q^4$ is $1/8^4=1/4096$. We use the coordinates $(a_1,a_2,b_1,b_2,c_1,c_2,d_1,d_2)$ for points in $Q^4$.

\newcommand{\trfour}{P_{4\text{-line}}}
\newcommand{\trthree}{P_{3\text{-line}}}
\newcommand{\itrfour}{P_{\square}}
\newcommand{\itrthree}{P_{\bigtriangleup}}
\newcommand{\winner}{P_{1 \ra \bigtriangleup}}
\newcommand{\loser}{P_{1 \leftarrow \bigtriangleup}}
\newcommand{\winlose}{P_{1,\bigtriangleup}}

The main computation that we need to do is to find the probability that $A \succ B \succ C \succ D \succ A$. Then using symmetry and the results already computed for three dice we will have enough to determine the probabilities of the four different tournament types. To see how that works, we introduce the following notation for the probabilities of the different isomorphism classes for three and four dice.

\begin{center}
\begin{tabular}{lc}
transitive 3-chain & $\trthree$ \\
intransitive 3-cycle & $\itrthree$ \\
transitive 4-chain & $\trfour$ \\
intransitive 4-cycle & $\itrfour$ \\
winner + 3-cycle & $\winner$ \\
loser + 3-cycle & $\loser$ \\
\end{tabular}
\end{center}

Generating  a random set of four dice and then deleting one of them at random gives a random set of three dice. If the four dice are a transitive chain, then the remaining three dice will be transitive. If the four dice are a 4-cycle, then removing one of them results in a 3-cycle or a 3-chain, with two possibilities for each. If the four dice are a 3-cycle plus winner/loser, then the result is a 3-cycle when the winner/loser is removed or a 3-chain if one of the other three dice is removed. Therefore,
\begin{align*}
   \trthree &= \trfour + \frac{1}{2} \itrfour +\frac{3}{4}\big(\winner + \loser \big) \\
   \itrthree &= \frac{1}{2} \itrfour +\frac{1}{4} \big(\winner + \loser \big) 
\end{align*}
 Furthermore, $\winner = \loser$, because the transformation 
 \[ A=(a_1,a_2,a_3) \mapsto A^*=(1-a_3,1-a_2,1-a_1)\]
reverses the relation between dice, and so the tournament associated to $(A^*,B^*,C^*,D^*)$ is the result of reversing all the edges in the tournament of $(A,B,C,D)$. Therefore, once we have found $\itrfour$, we use it and the already known $\trthree$ and $\itrthree$ to solve for the remaining probabilities. 
 
 \begin{lemma}
 $ \itrfour = 6 \, \pr(A \succ B \succ C \succ D \succ A) .$
 \end{lemma}
 \begin{proof}
 There are six ways to label the vertices of a 4-cycle.
 \end{proof}
A priori, there are 81 different polytopes in $Q^4$ to consider because each of the four relations in the cycle $A \succ B \succ C \succ D \succ A$ is $\succ_1, \succ_2$, or $\succ_3$. For $\sigma = \sigma_1\sigma_2\sigma_3\sigma_4$, where $\sigma_i \in\{1,2,3\}$, let $G_\sigma$ be the event
 $A \succ_{\sigma_1} B \succ_{\sigma_2} C \succ_{\sigma_3} D \succ_{\sigma_4} A$ as well as  the corresponding subset of $Q^4$. The union $G = \bigcup G_\sigma$ corresponds to the event $A \succ B \succ C \succ D \succ A$, and $\pr(G) = \sum \pr(G_\sigma)$. Fortunately, we can reduce the number of $G_\sigma$ whose volumes need to be computed. 
 
 \begin{lemma}
  If $\sigma'$ is a cyclic permutation of $\sigma$, then $\vol(G_{\sigma'})= \vol(G_\sigma)$.
 \end{lemma}
 \begin{proof}
   The map $Q^4 \ra Q^4: (A,B,C,D) \mapsto (B,C,D,A)$ is an isometry and maps $G_\sigma$ to $G_{\sigma'}$ where $\sigma=(\sigma_1,\sigma_2,\sigma_3,\sigma_4)$ and $\sigma'=(\sigma_2,\sigma_3,\sigma_4,\sigma_1)$.
 \end{proof}
 
\begin{lemma}
 If $\sigma$ has one or two distinct entries, then $G_\sigma = \emptyset$.
\end{lemma}
\begin{proof}
The $\sigma$ in question are the following and all cyclic permutations of them:
\begin{align*}
  &(iiii): 1111, 2222, 3333  \\
  &(iiij): 1112, 1113, 2221, 2223, 3331, 3332 \\
  &(iijj):  1122, 1133, 2233 \\
  &(ijij):  1212, 1313, 2323
 \end{align*}
For $\sigma$ in the first group the chain of inequalities obviously leads to a contradiction. Now let $\sigma = 1112$, which means that $a_1 < b_1 < c_1 < d_1$ and $d_2 < a_2$. But we also have $a_2 > b_2 > c_2 > d_2$, which is a contradiction. The remaining five $\sigma$ in the second group are disposed of similarly. From the third group we use the fact that the $\succ_i$ are transitive. If $\sigma = iijj$  then $G_\sigma$ is the event $A \succ_i B \succ_i C \succ_j D \succ_j A$, and this  implies that $A \succ_i C$ and $C \succ_j A$, which is impossible. Finally for $\sigma=ijij$ we consider  $A \succ_i B \succ_j C \succ_i D \succ_j A$. Then for $k \ne i, j$ we have $a_k > b_k > c_k > d_k >a_k$, which is impossible.
\end{proof}
%
 \begin{lemma}
 There are 36 different $\sigma$ for which $\pr(G_\sigma) >0$. Each is a cyclic permutation of one of the nine on the list below. If $\sigma'$ is a cyclic permutation of $\sigma$, then $\pr(G_{\sigma'})=\pr(G_\sigma)$.
 \begin{center}
 \begin{tabular}{cc}
 $\sigma$ & $\pr(G_\sigma)$ \\
 $1123$ & $229/322560$ \\
 $1132$ & $691507/294912000$ \\
$1213$ &  $40913/15482880$ \\
$1223$ & $5431/8064000$ \\
$1232$ & $32299/16515072$ \\
$1322$ & $38929/18432000$ \\
$1233$ & $229/322560$ \\
$1323$ & $40913/15482880$ \\
$1332$ & $691507/294912000$
 \end{tabular}
 \end{center}
 \end{lemma}
\begin{proof}
See the Appendix for the Sage code to compute the volumes and probabilities.
\end{proof}

Notice that $\pr(G_{1123})=\pr(G_{1233})$, $\pr(G_{1132})=\pr(G_{1332})$, and $\pr(G_{1213})=\pr(G_{1323})$. Consider the event $G_{1123}$, which means that
$A \succ_1 B \succ_1 C \succ_2 D \succ_3 A$. Apply the star operator to each die and verify that verify $A^* \succ_1 D^* \succ_2 C^* \succ_3 B^* \succ_3 A^*$. That is, $\succ_1$ reverses direction and changes to $\succ_3$, while $\succ_3$ reverses and changes to $\succ_1$, and $\succ_2$ only reverses direction. Thus, the isometry 
\[ Q^4 \ra Q^4: (A,B,C,D) \mapsto (A^*,D^*,C^*,B^*)\]
 maps $G_{1123}$ onto $G_{1233}$. It also maps $G_{1132}$ onto $G_{2133}$, which has the same volume as $G_{1332}$ because $2133$ and $1332$ are cyclically related. And it maps $G_{1213}$ onto $G_{1323}$.

\begin{lemma}
\[ \pr(A \succ B \succ C \succ D \succ A) = \pr(G)= \displaystyle \frac{99930571}{1548288000} \]
\end{lemma}
\begin{proof}
There are four cyclic permutations of each $\sigma$ in Lemma 9. So summing them and multiplying by four gives $\pr(G)$. 
\end{proof}

\begin{theorem} With four 3-sided dice in the balanced uniform model, the probabilities of the four isomorphism classes of tournaments are the following:
\begin{align*}
 \trfour &= \frac{110413771}{258048000}  \approx 0.42788  \quad \text{(transitive chain)}\\
 \\
 \itrfour &=  \frac{99930571}{258048000} \approx 0.38726   \quad \text{(intransitive 4-cycle)}\\
 \\
 \winner &=  \frac{23851829}{258048000} \approx 0.09243  \quad \text{(winner + 3-cycle)} \\
 \\
 \loser &= \frac{23851829}{258048000}  \approx 0.09243 \quad \text{(loser + 3-cycle)}
\end{align*}
\end{theorem}
\begin{proof}
With four vertices there are six different 4-cycles. From Lemma 9 we know $\pr(G)$, the probability of one particular 4-cycle, and so 
\[\itrfour = 6 \,\pr(G)= 6\, \bigg(\frac{99930571}{1548288000}\bigg) 
    = \frac{99930571}{258048000}  .\]
We also know $\trthree = 973/1280$ and $\itrthree=307/1280$ from Theorem 5, so that we can solve for $\trfour$, $\winner$, and $\loser$ using the equations
\begin{align*}
   \trthree &= \trfour + \frac{1}{2} \itrfour +\frac{3}{4}\big(\winner + \loser \big) \\
   \itrthree &= \frac{1}{2} \itrfour +\frac{1}{4} \big(\winner + \loser \big)  \\
   \winner  &= \loser  
\end{align*} 
\end{proof}
\pagebreak

\section{Sage Computation: Three Dice}

We define the polytopes (or ``polyhedra'' ) in Sage using the half-space representation \cite{SagePoly}, which is a system of linear inequalities of the form
\[ 0 \le \beta + \alpha_1 x_1 +\alpha_2 x_2 + \cdots +\alpha_n x_n. \]
This inequality is represented by $(\beta, \alpha_1,\alpha_2,\ldots,\alpha_n)$, a  tuple of length $n+1$. Our variables are $a_1,a_2,b_1,b_2,c_1,c_2$, so that, for example, the inequality $ 2a_1 + 4a_2 \le 3$ is represented by the $7$-tuple $(3,-2,-4,0,0,0,0)$. The twelve inequalities defining $Q^3$ and their Sage representations are the following:
 \begin{align*}
 &\quad & 0 &\le a_1,  &  a_1 &\le a_2, & 1&\le 2a_1 + 2a_2, & 2a_1 + 4a_2 &\le 3,  & &\quad\\
 &\quad & 0 &\le b_1,  & b_1 &\le b_2, & 1 &\le 2b_1 + 2b_2, & 2b_1 + 4b_2 &\le 3, &&\quad\\
 &\quad & 0 &\le c_1,  & c_1 &\le c_2,  & 1 &\le 2c_1 + 2c_2, & 2c_1 + 4c_2 &\le 3.  &&\quad
 \end{align*}
 \begin{align*}
 (0,1,0,0,0,0,0),\quad (0,-1,1,0,0,0,0) ,\quad  (-1,2,2,0,0,0,0),\quad(3,-2,-4,0,0,0,0) ,\\
(0,0,0,1,0,0,0),\quad  (0,0,0,-1,1,0,0), \quad (-1,0,0,2,2,0,0),\quad (3,0,0,-2,-4,0,0),\\
(0,0,0,0,0,1,0),\quad (0,0,0,0,0,-1,1),\quad (-1,0,0,0,0,2,2),\quad (3,0,0,0,0,-2,-4).
 \end{align*}
 
 To define $E_{123}$ corresponding to $A \succ_1 B \succ_2 C \succ_3 A$ there are nine additional inequalities.
\begin{align*}
 a_1 < b_1,  \quad  a_2 > b_2 ,\quad a_3 > b_3 \Leftrightarrow a_1 +a_2 < b_1+b_2 ,\\
 b_1 > c_1,  \quad b_2 < c_2,  \quad b_3 > c_3 \Leftrightarrow  b_1+b_2 < c_1 +c_2,\\
 c_1 > a_1,  \quad  c_2 > a_2, \quad  c_3 < a_3 \Leftrightarrow c_1 +c_2 > a_1 +a_2.
\end{align*}
The Sage representations are
\begin{align*}
       (0,-1,0,1,0,0,0), &\quad (0,0,1,0,-1,0,0), \quad (0,-1,-1,1,1,0,0), \\
      (0,0,0,1,0,-1,0), &\quad (0,0,0,0,-1,0,1), \quad (0,0,0,-1,-1,1,1), \\
      (0,-1,0,0,0,1,0), &\quad (0,0,-1,0,0,0,1),\quad  (0,-1,-1,0,0,1,1).
\end{align*}
The polytope $E_{132}$ corresponding to $A \succ_1 B \succ_3 C \succ_2 A$ has the additional inequalities represented by
\begin{align*}
      (0,-1,0,1,0,0,0), &\quad (0,0,1,0,-1,0,0), \quad (0,-1,-1,1,1,0,0), \\
      (0,0,0,1,0,-1,0), &\quad (0,0,0,0,1,0,-1), \quad (0,0,0,1,1,-1,-1), \\
      (0,-1,0,0,0,1,0), &\quad (0,0,1,0,0,0,-1), \quad (0,1,1,0,0,-1,-1).
\end{align*}

For the volume computations we call the integrate function in \emph{LattE integrale} using the Sage interface \cite{LattE}.

\newpage
\large
\renewcommand{\baselinestretch}{1.1}   
\normalsize

\begin{Verbatim}[frame=single,label=Sage input,fontfamily=courier]

from sage.interfaces.latte import integrate 

Q3_ieqs = [(0,1,0,0,0,0,0),(-1,2,2,0,0,0,0),(0,-1,1,0,0,0,0),
           (3,-2,-4,0,0,0,0),(0,0,0,1,0,0,0),(-1,0,0,2,2,0,0),
           (0,0,0,-1,1,0,0),(3,0,0,-2,-4,0,0),(0,0,0,0,0,1,0),
           (-1,0,0,0,0,2,2),(0,0,0,0,0,-1,1),(3,0,0,0,0,-2,-4)] 
Q3 = Polyhedron(ieqs = Q3_ieqs)
vol_Q3 = integrate(Q3.cdd_Hrepresentation(),cdd=True)

E123_ieqs = [(0,-1,0,1,0,0,0),(0,0,1,0,-1,0,0),(0,-1,-1,1,1,0,0),
             (0,0,0,1,0,-1,0),(0,0,0,0,-1,0,1),(0,0,0,-1,-1,1,1),
             (0,-1,0,0,0,1,0),(0,0,-1,0,0,0,1),(0,-1,-1,0,0,1,1)] 
E123 =  Polyhedron(ieqs = Q3_ieqs + E123_ieqs )
vol_E123 = integrate(E123.cdd_Hrepresentation(), cdd=True)
prob_E123 = vol_E123/vol_Q3

E213_ieqs = [(0,-1,0,1,0,0,0),(0,0,1,0,-1,0,0),(0,-1,-1,1,1,0,0),
            (0,0,0,1,0,-1,0),(0,0,0,0,1,0,-1),(0,0,0,1,1,-1,-1),
            (0,-1,0,0,0,1,0),(0,0,1,0,0,0,-1),(0,1,1,0,0,-1,-1)] 
E213 = Polyhedron(ieqs = Q3_ieqs + E213_ieqs)
vol_E213 = integrate(E213.cdd_Hrepresentation(), cdd=True)      
prob_E213 = vol_E213/vol_Q3

prob_E = 3*prob_E123 + 3*prob_E213

print("vol(Q3) =",vol_Q3)
print("P(E123) =",prob_E123,"=",n(prob_E123,digits=9))
print("P(E213) =",prob_E213,"=",n(prob_E213,digits=9))     
print("P(A>B>C>A) = P(E) =",prob_E,"=",n(prob_E,digits=9))
print("P(intransitive) =",2*prob_E,"=",n(2*prob_E,digits=9)) 
\end{Verbatim}

\begin{Verbatim}[frame=single,label=Output]
vol(Q3) = 1/512
P(E123) = 23/1800 = 0.0127777778
P(E213) = 3133/115200 = 0.0271961806
P(A>B>C>A) = P(E) = 307/2560 = 0.119921875
P(intransitive) = 307/1280 = 0.239843750
\end{Verbatim}

\large
\renewcommand{\baselinestretch}{1.2}   
\normalsize

\section{Sage Computation: Four Dice}

There are 16 inequalties needed to define $Q^4$ in Sage. Each is represented by a tuple of length 9. They are straightforward extensions of those used to define $Q^3$. 

To define $G_\sigma$ where $\sigma = ijkl$ we use the function $\texttt{G(i,j,k,l)}$ which returns the Sage polyhedron. This requires the inequalities for $Q^4$ and the inequalities for the four relations. 

For $m=1,2,3,4$ and $i=1,2,3$ the function $\texttt{g(m,i)}$ returns a list of three tuples asserting that the $m$th dominance relation is $\succ_i$. For example, $g(3,2)$ returns the tuples for $C \succ_2 D$ since the third relation is the one between $C$ and $D$. Then $G(i,j,k,l)$ simply concatenates the inequalities for $Q^4$ with those coming from $\texttt{g(1,i),g(2,j),g(3,k)}$, and $\texttt{g(4,l)}$, a total of 28 inequalties. The definition of $\texttt{g(m,i)}$ uses vectors $\texttt{a1,a2,a3},\ldots,\texttt{d1,d2,d3}$ so that an inequality such as $a_2 > b_2$ is represented by $\texttt{a2-b2}$. The vector $\texttt{a3}$ is actually $\texttt{-a1-a2}$ and likewise for $\texttt{b3,c3,d3}$, because an inequality such as $a_3 > b_3$ is equivalent to $-a_1-a_2 > -b_1 -b_2$.

The function $\texttt{pr(i,j,k,l)}$ returns the probability of $G_{ijkl}$. First it finds the dimension of $\texttt{G(i,j,k,l)}$. If the dimension is less than eight, then the polytope has zero volume and the probability is zero. If the dimension is eight, then the function uses the integrate function from \emph{LattE integrale} to find the volume and multiplies by $4096$ to get the probability. It is necessary to check the dimension first, because the integrate function does not return zero on polytopes of less than the full dimension.

\newpage
\large
\renewcommand{\baselinestretch}{1.05}   
\normalsize
\begin{Verbatim}[frame=single,label=Sage input,fontfamily=courier]

from sage.interfaces.latte import integrate 

Q4_ieqs = [(0,1,0,0,0,0,0,0,0),(-1,2,2,0,0,0,0,0,0),
           (0,-1,1,0,0,0,0,0,0),(3,-2,-4,0,0,0,0,0,0),
           (0,0,0,1,0,0,0,0,0),(-1,0,0,2,2,0,0,0,0),
           (0,0,0,-1,1,0,0,0,0),(3,0,0,-2,-4,0,0,0,0),
           (0,0,0,0,0,1,0,0,0),(-1,0,0,0,0,2,2,0,0),
           (0,0,0,0,0,-1,1,0,0),(3,0,0,0,0,-2,-4,0,0),
           (0,0,0,0,0,0,0,1,0),(-1,0,0,0,0,0,0,2,2),
           (0,0,0,0,0,0,0,-1,1),(3,0,0,0,0,0,0,-2,-4)] 
Q4 = Polyhedron(ieqs = Q4_ieqs)
vol_Q4 = integrate(Q4.cdd_Hrepresentation(),cdd=True)

a1=vector((0,1,0,0,0,0,0,0,0));
a2=vector((0,0,1,0,0,0,0,0,0));
a3=vector((0,-1,-1,0,0,0,0,0,0));
b1=vector((0,0,0,1,0,0,0,0,0));
b2=vector((0,0,0,0,1,0,0,0,0));
b3=vector((0,0,0,-1,-1,0,0,0,0));
c1=vector((0,0,0,0,0,1,0,0,0));
c2=vector((0,0,0,0,0,0,1,0,0));
c3=vector((0,0,0,0,0,-1,-1,0,0));
d1=vector((0,0,0,0,0,0,0,1,0));
d2=vector((0,0,0,0,0,0,0,0,1));
d3=vector((0,0,0,0,0,0,0,-1,-1));
def g(m,i):
    s = ((-1,1,1),(1,-1,1),(1,1,-1))
    if m==1:
        return [s[i-1][0]*(a1-b1),s[i-1][1]*(a2-b2),
                s[i-1][2]*(a3-b3)]
    elif m==2:
        return [s[i-1][0]*(b1-c1),s[i-1][1]*(b2-c2),
                s[i-1][2]*(b3-c3)]
    elif m==3:
        return [s[i-1][0]*(c1-d1),s[i-1][1]*(c2-d2),
                s[i-1][2]*(c3-d3)]
    elif m==4:
        return [s[i-1][0]*(d1-a1),s[i-1][1]*(d2-a2),
                s[i-1][2]*(d3-a3)]
end 
\end{Verbatim}
\begin{Verbatim}[frame=single,label=Sage input,fontfamily=courier]

def G(i,j,k,l):
   return Polyhedron(ieqs = Q4_ieqs+g(1,i)+g(2,j)+g(3,k)+g(4,l))
end

def pr(i,j,k,l):
   event = G(i,j,k,l)
   if event.dim() < 8:
       eturn 0
   else:
      return 4096*integrate(event.cdd_Hrepresentation(),cdd=True)
end

print(" Event         Probability")
print("G(1,1,2,3)    ",pr(1,1,2,3))
print("G(1,1,3,2)    ",pr(1,1,3,2))
print("G(1,2,1,3)    ",pr(1,2,1,3))
print("G(1,2,2,3)    ",pr(1,2,2,3))
print("G(1,2,3,2)    ",pr(1,2,3,2))
print("G(1,3,2,2)    ",pr(1,3,2,2))
print("G(1,2,3,3)    ",pr(1,2,3,3))
print("G(1,3,2,3)    ",pr(1,3,2,3))
print("G(1,3,2,2)    ",pr(1,3,2,2))
\end{Verbatim}
\begin{Verbatim}[frame=single,label=Output]
 Event         Probability 
G(1,1,2,3)     229/322560
G(1,1,3,2)     691507/294912000
G(1,2,1,3)     40913/15482880
G(1,2,2,3)     5431/8064000
G(1,2,3,2)     32299/16515072
G(1,3,2,2)     38929/18432000
G(1,2,3,3)     229/322560
G(1,3,2,3)     40913/15482880
G(1,3,2,2)     38929/18432000
\end{Verbatim}

\newpage

\end{document}